\documentclass[12pt]{amsart}

\usepackage{amsmath, amssymb, amscd, newlfont,caption,url}

\newcommand{\bbH}{{\mathbb{H}}}

\newcommand{\bbC}{{\mathbb{C}}}
\newcommand{\bbP}{{\mathbb{P}}}

\newcommand{\bbZ}{{\mathbb{Z}}}

\renewcommand{\div}{{\mathrm{div}}}

\numberwithin{equation}{section}

\newtheorem{Lem}[equation]{Lemma}

\newtheorem{Thm}[equation] {Theorem}
\newtheorem{Cor}[equation]{Corollary}

\title
[On a simple model of $X_0(N)$]
{On a simple model of $X_0(N)$}

\author{Iva Kodrnja}
\address{
Faculty of Civil Engineering, 
University of Zagreb,
Ka\v ci\' ceva 26, 10000 Zagreb,
Croatia}
 \email{ikodrnja@grad.hr}
\thanks{The  author acknowledges Croatian Science Foundation grant no. 9364.}
\begin{document}
\begin{abstract}
	We find plane models for all $X_0(N)$, $N\geq 2$. We observe a map from the modular curve $X_0(N)$ to the projective plane constructed using modular forms of weight $12$ for the group $\Gamma_0(N)$; the Ramanujan function $\Delta$, $\Delta(N\cdot)$ and the third power of Eisestein series of weight $4$, $E_4^3$, and prove that this map is birational equivalence for every $N\geq 2$.
	The equation of the model is the minimal polynomial of $\Delta(N\cdot)/\Delta$ over $\bbC(j)$.
\end{abstract}
\subjclass[2000]{11F11, 11F23}
\keywords{modular forms,  modular curves, birational equivalence, modular polynomial}
\maketitle

\section{Introduction}

The recent paper \cite{Muic2} by Mui\' c presents a new method of finding defining equations for modular curves. As an application of the method one example was presented - a map from $X_0(N)$ to the projective plane defined by 
\begin{equation}\label{m1}\mathfrak{a}_z\mapsto (\Delta(z):E_4^3(z):\Delta(Nz)),
\end{equation}

where $$E_4(z)=1+240\sum\limits_{n=1}^\infty \sigma_3(n)q^n$$
is the usual Eisenstein series and
$$\Delta(z)=q+\sum\limits_{n=2}^\infty \tau(n)q^n$$ 
is the Ramanujan delta function.

The image of the map (\ref{m1}) is an irreducible projective curve which we denote by $\mathcal{C}_N$. In \cite{Muic2}, Lemma 5-4 and Lemma 5-5, it is proved that the curve $\mathcal{C}_p$ is birational to $X_0(p)$ for every prime number $p$ and a question was posed whether the map is always birational. 

We answer that question and prove that this is true for every number $N\geq 2$.

\begin{Thm}\label{L1}
The curve $\mathcal{C}_N$ is birational to $X_0(N)$ for every $N\geq 2$.	
\end{Thm}

In the proof we use ideas from \cite{ishida}. 
\begin{Cor}
Modular functions $j$ and $\Delta(N\cdot)/\Delta$ generate $\bbC(X_0(N))$.
\end{Cor}

This result gives us simple models for all $X_0(N)$. Some equations can be found in Section \ref{pro}. We must mention that, although simpler than the classical modular equation, these models have the same deficiencies - the degrees are big (equal $\psi(N)$) and the coefficients of equations are rather large. As computed in \cite{comp}, the explicit bound for the logarithmic height of the modular polynomial of prime level $l$ is $$6l\log(l)+18l.$$
It would be interesting to compute the height of our polynomials, and we believe that the bound would be very close to this one.

The classical modular equation, the minimal polynomial of $j(N\cdot)$ over $\bbC(j)$ is very hard to compute (\cite{BKS}) and so are other modular polynomials for different modular functions (\S 7 of \cite{BKS} or \cite{rac}) and this is also the case for our polynomials.

Computing the defining equations of modular curves is an interesting problem in the theory of modular curves. We would like to mention the paper \cite{yy} by Young which provides the simplest equations for $X_0(N)$ and as well for $X(N)$ and $X_1(N)$. There are also interesting papers by Ishida and Ishii , \cite{ishida}, \cite{ii} , where the authors find generators for function fields of $X(N)$ and $X_1(N)$. We use their argument in proving Theorem \ref{L1}. 

The method we use for finding models is developed by Mui\'c in \cite{Muic1},\cite{MuMi} and \cite{Muic2}.



I would like to thank G. Mui\'c for introducing me into this very interesting subject and for many useful conversations and advices.

\section{Maps to projective plane}\label{maps}


Let $\Gamma$ be a Fuchsian group of first order.
The quotient space of the complex upper half-plane $\bbH$ by the action of $\Gamma$ is a Riemann surface which we will denote $X(\Gamma)$. This set can be compactified by adding orbits of cusps of $\Gamma$. For $\Gamma=\Gamma_0(N)$, this compact Riemann surface is denoted by $X_0(N)$ and called modular curve.

The main idea of \cite{Muic2} is to map $X(\Gamma)$ to the projective plane $\bbP^2$ using modular forms. It is achieved in the following way:
 
Select $k\geq 2$ such that $\dim M_k(\Gamma)\geq 3$. Take three linearly independent modular forms $f$, $g$ and $h$ in $M_k(\Gamma)$ and construct the map $X(\Gamma)\mapsto \bbP^2$ by defining it on the complement of points in $X(\Gamma)$ which are orbits of common zeros of $f$, $g$ and $h$ by
\begin{equation}\label{map}
\mathfrak{a}_z\mapsto(f(z):g(z):h(z)).
\end{equation}   
The map defined in this way is uniquely determined holomorphic map from the Riemann surface $X(\Gamma)$ to $\bbP^2$. It is actually a rational (in fact regular because the domain is compact) map
$$\mathfrak{a}_z\mapsto(1:g(z)/f(z):h(z)/f(z)).$$

The image is an irreducible projective curve which we denote by $\mathcal{C}(f,g,h)$, whose degree is less or equal to $\dim M_k(\Gamma)+g(\Gamma)-1$. This bound for $\deg(\mathcal{C}(f,g,h))$ equals the degree of integral divisors attached to modular forms $f$,$g$ and $h$ (see \cite{Muic2}, Lemma 2-2 (vi)) and can be shown by calculating the number of points in the intersection of $\mathcal{C}(f,g,h)$ with a line in general position.

The degree of the map (\ref{map}) is the degree of the field extension 
$$\bbC(\mathcal{C}(f,g,h))\subset \bbC(X(\Gamma)),$$
and we denote it by $d(f,g,h)$. 

The field of rational functions $\bbC(\mathcal{C}(f,g,h))$ of the image curve is isomorphic to a subfield of $\bbC(X(\Gamma))$ generated over $\bbC$ by $g/f$ and $h/f$. Therefore, the map (\ref{map}) is birational equivalence if and only if $g/f$ and $h/f$ generate $\bbC(X(\Gamma))$.

In \cite{Muic2}, the following formula for the degree of the image curve $\mathcal{C}(f,g,h)$ was proved:
\begin{Thm}\label{formula}
Assume that $k\geq 2$ is an integer such that $\dim(M_k(\Gamma))\geq 3$. Let $f$, $g$, $h\in M_k(\Gamma)$ be three linearly independent modular forms. Then, we have the following:
$$d(f,g,h)\deg\mathcal{C}(f,g,h)=\dim(M_k(\Gamma))+g(\Gamma)-1-\sum\limits_{\mathfrak{a}\in X(\Gamma)}\min(\mathfrak{c}'_f(\mathfrak{a}),\mathfrak{c}'_g(\mathfrak{a}),\mathfrak{c}'_h(\mathfrak{a})),$$ 
where $\mathfrak{c}'_f$,$\mathfrak{c}'_g$ and $\mathfrak{c}'_h$ are integral divisors attached to modular forms $f$,$g$ and $h$.
\end{Thm}

\section{Proof of Theorem \ref{L1}}\label{pro}

For a non-constant function $f\in \bbC(X(\Gamma))$, the degree of the subfield generated by $f$ equals the degree of the divisor of poles of $f$ (see \cite{Miranda}, \S 6) which we will denote by $$d(f)=\deg(\div_\infty(f))=[\bbC(X(\Gamma)):\bbC(f)].$$


Returning to the map (\ref{map}) we have an easy condition for birational equivalence: 

\begin{Lem}\label{L2}
	The map (\ref{map}) is a birational equivalence if $$\gcd(d(g/f),d(h/f))=1.$$\end{Lem} 

The converse is not true. For the map (\ref{m1}) when $N>2$ the degrees of divisors of poles of $j$ and $\Delta(N\cdot)/\Delta$ are always divisible by $2$.

But we can look at other functions in $\bbC(\mathcal{C}(f,g,h))$. This is an argument which is used in \cite{ishida} (see Lemma 2) to find generators of function fields of modular curves $X(N)$ and $X_1(N)$. 

\begin{Lem}\label{L3}
	If there are two non-constant functions $f_1$ and $f_2$ in $\bbC(g/f,h/f)$ such that $\gcd(d(f_1),d(f_2))=1$, then the map (\ref{map}) is a birational equivalence.	
	
\end{Lem}  
\begin{proof}
	Since $f_1\in \bbC(g/f,h/f)$ we have a sequence
	$$\bbC(f_1)\subseteq \bbC(g/f,h/f)\subseteq \bbC(X(\Gamma)),$$
	
	and from this we see that $d(f,g,h)$ divides $d(f_1)$. 
	
	If $\gcd(d(f_1),d(f_2))=1$, then $d(f,g,h)=1$.
\end{proof}

Now we can prove Theorem \ref{L1}.

\begin{proof}
We look at the following two non-constant functions in $\bbC\left(j,\Delta(N\cdot)/\Delta\right)$:
$$f_1=j \qquad \text{and}\qquad f_2=j^{N-2}+\left(\frac{\Delta(N\cdot)}{\Delta}\right)^{N-1}.$$

We compute $d(f_1)$ and $d(f_2)$ and show that these numbers are relatively prime.

First, we need the divisors of modular forms $\Delta$, $\Delta(N\cdot)$ and $E_4^3$. They are calculated in \cite{Muic2}, Lemma 4-3. For our purpose, it is important that $E_4^3$ has zeros in the $\Gamma_0(N)$-orbits of $(1+\sqrt{(-3)})/2$ and that $\Delta$ and $\Delta(N\cdot)$ have zeros at cusps of $\Gamma_0(N)$ so supports of their divisors are disjoint.

The full set of representatives of cusps of $\Gamma_0(N)$ is the set of rational numbers $c/d$ where $d$ is a positive divisor of $N$, $\gcd(c,d)=1$ and there are $\varphi(\gcd(d,N/d))$ representatives with denominator $d$ ($\varphi$ denotes the Euler function).

Divisors of $\Delta$ and $\Delta(N\cdot)$ can also be easily calculated using the formula for the order of an $\eta-$quotient at the cusp $c/d$ (see \cite{ligozat}). Here are their divisors:
 
\begin{align*}
&\div(\Delta)=\sum\limits_{\substack{d|N \\ 1\leq d\leq N}} \frac{N}{d}\frac{1}{\gcd(d,N/d)}\mathfrak{a}_{c/d}\\
&\div(\Delta(N\cdot))=\sum\limits_{\substack{d|N \\ 1\leq d\leq N}} \frac{d}{\gcd(d,N/d)}\mathfrak{a}_{c/d}.
\end{align*}

Now, the divisor of poles of $f_1$ is minus the divisor of $\Delta$ and its degree is

\begin{equation}\label{dj}
d(f_1)=\sum\limits_{\substack{d|N \\ 1\leq d\leq N}} \frac{N}{d}\frac{\varphi(\gcd(d,N/d))}{\gcd(d,N/d)}=\psi(N),
\end{equation} 

where $\psi(N)$ is the Dedekind psi function, $\psi(N)=[SL_2(\bbZ):\Gamma_0(N)]$ is the degree of all divisors of modular forms of weight $12$ on $\Gamma_0(N)$.

Let us compute the divisor of $f_2$. 
Function $\Delta(N\cdot)/\Delta$ has poles at the cusps $c/d$ for $d-N/d<0$, that is for $d<\sqrt{N}$. We have 

\begin{equation*}
\div_\infty(\Delta(N\cdot)/\Delta)=\sum\limits_{\substack{d|N \\ 1\leq d\leq \sqrt{N}}} \frac{N/d-d}{\gcd(d,N/d)}\mathfrak{a}_{c/d}.
\end{equation*}
  


The function $j^{N-2}$ has poles at all cusps of $\Gamma_0(N)$ and we conclude that $f_2$ has poles at all cusps. In the cusps $c/d$ where $d\geq \sqrt{N}$ the order of pole equals the order of pole of $j^{N-2}$ whereas in the cusps $c/d$ for $d<\sqrt{N}$ the order of pole is the greater of orders of poles of  $j^{N-2}$ and $(\Delta(N\cdot)/\Delta)^{N-1}$. Hence we have 
\begin{align}\label{e34}
d(f_2)=&\sum\limits_{\substack{d|N \\ 1\leq d\leq \sqrt{N}}} \frac{\varphi(\gcd(d,N/d))}{\gcd(d,N/d)}\max{\left(\frac{N}{d}(N-2),\left(\frac{N}{d}-d\right)(N-1)\right)}\\
& + \sum\limits_{\substack{d|N \\ d\geq \sqrt{N}}} \frac{\varphi(\gcd(d,N/d))}{\gcd(d,N/d)}\frac{N}{d}(N-2).\nonumber
\end{align}

The maximum appearing in formula (\ref{e34}) equals $\frac{N}{d}(N-2)$ for $d>1$ and for $d=1$ the maximum is $(N-1)^2=N(N-2)+1$ and we have 

\begin{equation}\label{e35}
d(f_2)=N(N-2)+1+ \sum\limits_{\substack{d|N \\ 1<d\leq N}} \frac{\varphi(\gcd(d,N/d))}{\gcd(d,N/d)}\frac{N}{d}(N-2)=(N-2)d(f_1)+1.
\end{equation}

Since $\gcd(d(f_1),d(f_2))=1$, Lemma \ref{L3} implies that $d_N=1$ and we have proved Theorem \ref{L1}. 

\end{proof}



The degree of the map (\ref{m1}) is one and the degree of $\mathcal{C}_N$ can be computed via Theorem \ref{formula}. We have:
$$\deg \mathcal{C}_N=\dim(M_{12}(\Gamma_0(N)))+g(\Gamma_0(N))-1=\psi(N).$$

We present few equations (the polynomial $P_N$ is the minimal polynomial of $\Delta(N\cdot)/\Delta$ over $\bbC(j)$, $P_N(j,\Delta(N\cdot)/\Delta)=0$):
\begin{flushleft}
\begin{align*}
P_2(x,y)=&\newcommand{\Bold}[1]{\mathbf{#1}}16777216  y^{3} - x y + 196608 y^{2} + 768  y + 1\\
P_3(x,y)=&\newcommand{\Bold}[1]{\mathbf{#1}}150094635296999121 y^{4} - x^{2} y + 38263752 \, x y^{2} - 213516729579636 \, y^{3} + 1512 \, x y \\ &+ 10589493366 \, y^{2} - 177876 \, y + 1\\
P_4(x,y)=&\newcommand{\Bold}[1]{\mathbf{#1}}324518553658426726783156020576256 \, y^{6} - 4096 \, x^{3} y^{2} + 6597069766656 \, x^{2} y^{3}\\ & - 2490310449950789468160 \, x y^{4} + 193118646128519322884263378944 \, y^{5} - x^{3} y \\&+ 1620049920 \, x^{2} y^{2} - 569986827839078400 \, x y^{3} + 38322004008487170909143040 \, y^{4} \\ &+ 2256 \, x^{2} y + 9349606932480 \, x y^{2} + 2538589037956201185280 \, y^{3} - 1105920 \, x y \\&+ 557658553712640 \, y^{2} + 40894464 \, y + 1\\
P_5(x,y)=&\newcommand{\Bold}[1]{\mathbf{#1}}867361737988403547205962240695953369140625 \, y^{6} - x^{4} y + 29296875000 \, x^{3} y^{2}\\ & - 246763229370117187500 \, x^{2} y^{3} + 547152012586593627929687500000 \, x y^{4}\\& - 85798035343032097443938255310058593750 \, y^{5} + 3000 \, x^{3} y + 1243896484375000 \, x^{2} y^{2} \\ &+ 12913942337036132812500000 \, x y^{3} + 2829028744599781930446624755859375 \, y^{4}\\& - 2587500 \, x^{2} y + 1322387695312500000 \, x y^{2} - 31095165759325027465820312500 \, y^{3}\\& + 587500000 \, x y + 29664516448974609375 \, y^{2} - 9433593750 \, y + 1\\
\end{align*}
\end{flushleft}









\begin{thebibliography}{999999}

\bibitem{rac} {\sc I.~Blake, J.~A.~Csirik, M.~Rubinstein, G.~Seroussi,}{\em On the computation of modular polynomials for elliptic curves}, Thec. Report, Hewlett-Packard Laboratories, (1999), \url{http://www.math.uwaterloo.ca/~mrubinst/publications/publications.html}

\bibitem {comp} {\sc R.~Br\" oker,  A.~V.~Sutherland,} {\em An explicit height bound for the classical modular polynomial}, Ramanujan Journal {\bf 22} (2010), 293--313.

\bibitem{BKS} {\sc R.~Br\" oker, K.~Lauter, A.~V.~Sutherland,}{\em    Modular polynomials via isogeny volcanoes,} 
Mathematics of Computation {\bf 81} (2012), 1201--1231.







\bibitem{ishida}
{\sc N.~Ishida,} {\em Generators and equations for modular function fields of principal congruence subgroups,} Acta Arithmetica, {\bf 85} (1998) no 3, 197--207.

\bibitem{ii}
{\sc N.~Ishida, N.~Ishii},{\em Generators and defining equations of the modular function field of the group $\Gamma_1(N)$,} Acta Arithmetica, {\bf 101} (2002) no 4, 303--320.

\bibitem{ligozat} {\sc G.~Ligozat,} {\it Courbes modulaires de genre $1$}, Bull. Soc. Math. France[Memoire 43] (1972), 1-80.

\bibitem{Miyake}
{\sc T.~Miyake,} {\em Modular forms,} Springer-Verlag (2006).


\bibitem{Miranda} {\sc R.~Miranda,}{\em Algebraic Curves and Riemann Surfaces,} 
Graduate Studies in Mathematics {\bf 5} (1995).

\bibitem{Muic} {\sc G.~Mui\' c}, 
{\em Modular curves and bases for the spaces of cuspidal modular forms,}
Ramanujan J. {\bf 27} (2012), 181–-208.


\bibitem{Muic1} {\sc G.~Mui\' c,} {\em On embeddings of curves in projective spaces,} 
Monatsh. Math. {\bf  Vol. 173, No. 2} (2014), 239--256.

\bibitem{MuMi} {\sc G.~Mui\' c, D.~Miko\v c,} { Birational maps of $X(1)$ into $\mathbb P^2$,} Glasnik Matematicki {\bf Vol. 48}, No. 2 (2013), 301--312.

\bibitem{Muic2} {\sc G.~Mui\' c,} {\em On degrees and birationality of the maps $X_0(N)\rightarrow \mathbb P^2$ constructed via modular forms, } 
Monatsh. Math. {\bf  Vol. 180, No. 3} (2016), 607--629.


\bibitem{shi} {\sc G.~Shimura,} {\em  Introduction to the arithmetic theory of automorphic functions. Kanô Memorial 
Lectures,} No. 1. Publications of the Mathematical Society of Japan, No. 11. Iwanami Shoten, Publishers, Tokyo; Princeton 
University Press, Princeton, N.J., 1971. 







\bibitem{yy} {\sc Y.~Yifan,} {\em Defining equations of modular curves,} Advances in Mathematics {\bf 204} (2006) 481-–508.






\end{thebibliography}
\end{document}